\documentclass[12pt,letterpaper,english]{amsart}
\usepackage[letterpaper,textwidth=6in,textheight=8.75in,centering]{geometry}

\usepackage[utf8]{inputenc}
\usepackage[T1]{fontenc}
\usepackage{lmodern}
\usepackage{babel}
\usepackage{amsmath,amsthm,amssymb}
\usepackage{enumitem,microtype,xurl}
\usepackage[pdfusetitle,bookmarks=true,bookmarksnumbered=false,
 bookmarksopen=true,bookmarksopenlevel=1,pdfborder={0 0 1},
 backref=false,colorlinks=false]{hyperref}
\hypersetup{bookmarksdepth=2,
 pdftitle={Spectral G\r{a}rding Elliptic Equations},
 pdfauthor={Hao Fang and Biao Ma}}
\allowdisplaybreaks
\numberwithin{equation}{section}

\theoremstyle{plain}
\newtheorem{theorem}{Theorem}[section]
\newtheorem{proposition}[theorem]{Proposition}
\newtheorem{lemma}[theorem]{Lemma}
\newtheorem{corollary}[theorem]{Corollary}

\theoremstyle{definition}
\newtheorem{definition}[theorem]{Definition}
\newtheorem{example}[theorem]{Example}
\theoremstyle{remark}
\newtheorem{remark}[theorem]{Remark}
\newtheorem*{remark*}{Remark}

\newcommand{\one}{\mathbf1}
\newcommand{\ddc}{dd^c}

\newcommand{\C}{\mathcal C}

\title[Spectral G\AA rding elliptic equations]
{Spectral G\aa rding Elliptic Equations}
\author{Hao Fang}
\address{Department of Mathematics, University of Iowa, Iowa City, IA 52246, USA}
\email{hao-fang@uiowa.edu}
\thanks{H.~F.'s work is partially supported by a Simons Foundation
mathematics collaboration grant.}
\author{Biao Ma}
\address{School of Mathematical Sciences, East China Normal University,
500 Dongchuan Road, Shanghai, China}
\email{bma@math.ecnu.edu.cn}
\thanks{B.~M. is partially supported by NSFC grant No.~12471052.}
\date{}
\begin{document}
\begin{abstract}
We study spectral G\aa rding elliptic equations on compact K\"ahler
manifolds. Smooth solvability is equivalent to the existence of a smooth
representative in the admissible domain. The concavity theorem for ideal
G\aa rding polynomials gives a fixed concave formulation to which
Sz\'ekelyhidi's second-order argument applies in every degree
$2\le d\le n$, extending Lin's top-degree solvability theorem.
Combined with the numerical criteria of
Fu--Li--Zhang--Zhang and Chen--Nie--Xu, this yields a Nakai--Moishezon
type criterion for solvability on compact K\"ahler manifolds.

A solution converts the local ideal G\aa rding structure into a global
Lorentzian structure. Its mixed derivatives give a strict nonlinear
Hodge--Riemann hierarchy from the degree of the equation down to degree
two, and the homogenization of the intersection polynomial is Lorentzian in nef class directions. Numerical criteria then produce explicit
Hodge--Riemann classes, Lefschetz isomorphisms in bidegree $(1,1)$, and
weighted Hodge-index inequalities.
\end{abstract}
\maketitle

\section{Introduction}\label{introsection}

In this paper, we study a class of spectral G\aa rding elliptic equations on compact K\"ahler manifolds.
In the two preceding papers,
joint with Wu, we established Liouville rigidity for degenerate Hessian
equations \cite{Liouville} and the $L^\infty$ estimate through
determinantal capacity \cite{Linfty} for G\aa rding elliptic operators. This paper completes the sequence for the spectral case.
 We prove the second-order estimate and smooth
solvability of spectral G\aa rding equations in every degree
$2\le d\le n$, with a Nakai--Moishezon type
criterion on compact K\"ahler manifolds as a corollary. Then we return from the PDE
to algebra, showing that a solution turns the local ideal G\aa rding
structure of the PDE into a global Lorentzian polynomial structure on the
cohomological level.

 Our theory of G\aa rding and ideal G\aa rding
polynomials \cite{Garding,Ideal} is the local algebraic framework of our study of general Hessian equations. It is inspired by the classical
Caffarelli--Nirenberg--Spruck theory  \cite{CNS},
in which both ellipticity and concavity of the operator are required.
The ideal G\aa rding theory provides the algebraic foundation for
our study: the admissible domain, the ellipticity and the
concavity are properties of the polynomial. In \cite{Linfty} we
introduced the corresponding G\aa rding elliptic operators, that is,
polynomials on the Hermitian forms that are G\aa rding with respect to the
cone of semipositive forms. The corresponding PDEs
carry good algebraic properties and include many classical
examples, among them the complex Hessian and Hessian quotient equations,
the $J$-equation and the supercritical deformed Hermitian--Yang--Mills
equation. We call such an operator \emph{ideal} when the G\aa rding components of its
iterated derivatives are convex; equivalently, its positive scalarizations
are ideal G\aa rding polynomials. The spectral
operators belong to this class.

An equation is \emph{spectral} when the complex Hessian
enters only through its spectrum. The operators considered here are
polarizations of univariate polynomials, evaluated at the eigenvalues
of $A_\omega=\chi^{-1}\omega$. Classical examples include the complex
Monge--Amp\`ere equation \cite{Yau78}, the complex $k$-Hessian equation \cite{HMW,DK,Szekelyhidi,Zhang17},
the Hessian quotient equations \cite{Szekelyhidi}, the $J$-equation
\cite{Chen00,Wei06,SW08},
the inverse $\sigma_k$ equations \cite{FLM11}, and the supercritical
deformed Hermitian--Yang--Mills equation \cite{LYZ,CJY}.
The equations of top degree $d=n$, the
general inverse $\sigma_k$ equations, were first studied by Lin
\cite{LinConvex,LinCorr,Lin}. His work on right-Noetherian polynomials,
whose defining feature is the monotone root sequence recalled in
Section~\ref{intro:setup}, predates our G\aa rding polynomial theory and was a
motivation for it; right-Noetherian polynomials are exactly the univariate
G\aa rding polynomials \cite[Section~9.2]{Garding}. Lin established smooth solvability under a
subsolution condition \cite[Theorem~1.4]{Lin}. In this paper we treat the general degree $2\le d\le n$ on compact
K\"ahler manifolds. Further history is recalled in
Subsection~\ref{intro:history}.

\subsection{Spectral G\aa rding equations and the main theorem}
\label{intro:setup}

Let $(X^n,\chi)$ be a compact connected K\"ahler manifold,
and let $\omega_0$ be a smooth closed real $(1,1)$-form.
We follow the convention of \cite{FMForms} and consider a closed mixed-degree real form
\[
\Lambda=\sum_{k=0}^n\Lambda^{[k]},
\] where the superscript
$[k]$ denotes the $(k,k)$-component of a mixed-degree form.
 For a
$(1,1)$-form $\eta$, write
$e^\eta=\sum_{j=0}^n\eta^j/j!$.
We study the fully nonlinear equation
\begin{equation}\label{equation}
 \bigl(\Lambda\wedge e^{\omega_u}\bigr)^{[n]}=0,
 \qquad \omega_u=\omega_0+dd^cu,
\end{equation}
for a smooth real-valued function $u$, with a fixed smooth closed real form $\omega_0$.
In this paper, we focus on the spectral case
\begin{equation}\label{Lambda-def}
 \Lambda=\sum_{j=0}^d
 a_j\frac{\chi^{n-j}}{(n-j)!},
\end{equation}
where $2\le d\le n$, $a_1,\ldots,a_d$ are real
constants, $a_d>0$, and
$a_0\in C^\infty(X,\mathbb R)$. The equation depends only on the eigenvalues of $A_{\omega_u}:=\chi^{-1}\omega_u$.  In fact,  let $\sigma_j$
denote the $j$th elementary symmetric function,
with $\sigma_0\equiv1$, and set
\begin{equation}\label{polynomial}
 P_x(\lambda)=\sum_{j=0}^d a_j\sigma_j(\lambda),
 \qquad \lambda\in\mathbb R^n.
\end{equation}
Then
\eqref{equation} becomes the spectral equation
\[
P_x(\lambda(A_{\omega_u}))=0.
\]
For a Hermitian
matrix $A$, we will write $P_x(A)=P_x(\lambda(A))$. 

Put $\one=(1,\ldots,1)\in\mathbb R^n$. The diagonal restriction
\[
 p_x(s):=P_x(s\one)
\]
is a univariate polynomial in $s$. We say a univariate polynomial
of degree $m$ has a \emph{monotone root sequence}
if each of its derivatives of order at most $m-1$ has a real root and their
largest roots are nonincreasing. This is Lin's right-Noetherian condition,
equivalent to the univariate G\aa rding condition in our theory
\cite[Section~9.2]{Garding}. For $0\le k\le d-1$, let $\rho_k$ be the largest
real root of $p_x^{(k)}$. For $k\ge1$, $\rho_k$ is independent of $x$.
Our main analytic hypothesis is the \emph{strict first-gap condition}:
\begin{equation}\label{strict}
 \rho_0(x)>\rho_1\ge\rho_2\ge\cdots\ge\rho_{d-1}
 \qquad(x\in X).
\end{equation}
Thus $p_x(s)$ has a monotone root sequence and
 has a strict first root gap at every $x$. This is Lin's
strictly right-Noetherian condition \cite[Definition~2.6]{Lin}.
By the polarization theorems of \cite[Section~9.2]{Garding} and
\cite[Theorems~8.1--8.2]{Ideal},
\eqref{strict} makes $P_x(\lambda)$ an ideal G\aa rding polynomial. 

For a G\aa rding
polynomial $G$, let $\C_G$ denote its G\aa rding component. The admissible
domain is the first derived G\aa rding component
\cite[Definition~8.4 and Proposition~5.5]{Garding}:
\begin{equation}\label{domains}
 \mathcal A=\C_{P_x}^{(1)}=\bigcap_{i=1}^n\C_{\partial_iP_x}.
\end{equation}
It is independent of $x$, convex, and invariant under positive
semidefinite increments, and $P_x$ is elliptic there; see
Section~\ref{subsolutionsection}. We also use $\mathcal A$ for its spectral
lift to Hermitian matrices.

\begin{theorem}\label{main}
Let $\omega_0$ be a closed smooth real $(1,1)$-form with $\alpha=[\omega_0]$. Assume \eqref{strict} and the compatibility condition
\begin{equation}\label{compatibility}
 \int_X\bigl(\Lambda\wedge e^{\omega_0}\bigr)^{[n]}=0.
\end{equation}
Then the following are equivalent:
\begin{enumerate}[label=\textup{(\roman*)}]
\item There is a smooth closed representative
$\underline\omega=\omega_0+dd^c\underline u$ such that
$A_{\underline\omega}(x)\in\mathcal A$ for every $x\in X$.
\item There is a smooth admissible $\omega=\omega_0+dd^cu$
satisfying
\begin{equation}\label{target}
 P_x(A_\omega)=0.
\end{equation}
\end{enumerate}
$u$ is unique modulo constants. 
\end{theorem}
After a variable
shift, the case $d=n$ of Theorem~\ref{main} is Lin's solvability theorem
\cite[Theorem~1.4]{Lin}. For $d<n$, the theorem treats general
$\sigma_k$ equations beyond the previously established Hessian
and Hessian quotient cases
\cite[Corollary~3 and Section~7]{Szekelyhidi}. The proof is independent
of Lin's theorem.

For $P_x=\sigma_k-f(x)$ with $f(x)>0$ and
$2\le k<n$, the admissible domain satisfies
$\Gamma_k\subsetneq\mathcal A\subset\Gamma_{k-1}$, where $\Gamma_k=\{\sigma_1>0,\ldots,\sigma_k>0\}$, so condition
\textup{(i)} is a priori weaker than the existence of a $k$-positive
representative. Under the compatibility hypothesis with $f>0$, the
solution lies in $\Gamma_k$, so the two conditions are in fact
equivalent: a class containing a representative in $\mathcal A$
contains a $k$-positive representative. Compare the existence theory in
$k$-positive classes \cite{Sun17,Szekelyhidi}.

\subsection{A numerical criterion for solvability}\label{intro:numerical}

Chen--Nie--Xu \cite{CNX} characterized, on projective manifolds, the
smooth representatives in the G\aa rding component of a strictly
right-Noetherian polynomial by uniform intersection inequalities on
subvarieties. Fu--Li--Zhang--Zhang \cite{FLZZ} obtained such a
characterization on compact K\"ahler manifolds, in terms of mixed
intersection numbers with one auxiliary K\"ahler class. For a general
polynomial of degree $d<n$ both results stop at a representative in the G\aa rding component of the
polynomial, or of a slight shift of it, which lies in $\mathcal A$, while
a separate existence theorem is required. 
Theorem~\ref{main} provides exactly that.

\begin{corollary}\label{cor:numerical-solvability}
Assume the hypotheses of Theorem~\ref{main}.
The following are equivalent:
\begin{enumerate}[label=\textup{(\roman*)}]
\item There is a K\"ahler class $\kappa$ such that, for every irreducible
analytic subvariety $Y\subset X$ of dimension $k$, including $Y=X$, and
every $n-d\le q\le\min\{k,n-1\}$,
\begin{equation}\label{mixed-numerical}
 \int_Y\kappa^{k-q}\wedge\bigl(\Lambda\wedge e^\alpha\bigr)^{[q]}>0 .
\end{equation}
\item There is a smooth
$\underline\omega\in \alpha$ with
$A_{\underline\omega}(x)\in\mathcal A$ for every $x\in X$.
\item Equation~\eqref{target} has a smooth admissible solution in
$\alpha$.
\end{enumerate}
 When these conditions
hold, there is $\delta_{\mathrm{num}}>0$ such that, for every K\"ahler class
$\kappa$,
\begin{equation}\label{uniform-numerical}
 \int_Y\kappa^{k-q}\wedge\bigl(\Lambda\wedge e^\alpha\bigr)^{[q]}
 \ge\delta_{\mathrm{num}}\int_Y\kappa^{k-q}\wedge [\chi]^q
\end{equation}
for all $Y$ and $q$ as in \textup{(i)}. 
\end{corollary}

\subsection{From PDE to Lorentzian polynomials}\label{intro:nhr}

A solution also determines Hodge--Riemann pairings and a Lorentzian
intersection polynomial. These conclusions extend to equations
\begin{equation}\label{general-equation}
 \bigl(\Lambda\wedge e^\omega\bigr)^{[n]}=0,
\end{equation}
where $\Lambda=\sum_{k=n-d}^n\Lambda^{[k]}$, $2\le d\le n$, and each
$\Lambda^{[k]}$ is a smooth real $(k,k)$-form, closed for $k<n$.
Let 
\begin{equation}\label{cohomological-polynomial}
 \mathcal I(\gamma)=\int_X\bigl(\Lambda\wedge e^\gamma\bigr)^{[n]},
 \qquad \gamma\in H^{1,1}(X,\mathbb R).
\end{equation}

Lorentzian polynomials were introduced by Br\"and\'en and Huh
\cite{NHRBH} and also by Anari--Oveis Gharan--Vinzant \cite{AOGV}, Anari--Liu--Oveis Gharan--Vinzant \cite{ALOGV2,ALOGV}.

\begin{theorem}\label{mainNHR}
Let $\alpha=[\omega_0]\in H^{1,1}(X,\mathbb R)$. 
Suppose that either
\begin{itemize}
\item 
$\Lambda$ determines a spectral equation satisfying the hypotheses of Theorem~\ref{main} and one of the equivalent conditions in Corollary~\ref{cor:numerical-solvability} holds; or
\item there exists $\omega\in \alpha$ that solves \eqref{general-equation} and satisfies
\ref{H1}--\ref{H2} of Section~\ref{nhr:section}.
\end{itemize}
Then the following hold.
\begin{enumerate}[label=\textup{(\alph*)}]
\item\label{mainNHR-a}
For K\"ahler classes $\beta_1,\ldots,\beta_j$,
$0\le j\le d-2$, the pairing
\begin{equation}\label{intro-NHR}
(\gamma,\eta)\longmapsto
\int_X\gamma\eta\,\beta_1\cdots\beta_j
\bigl(\Lambda\wedge e^\alpha\bigr)^{[n-j-2]}
\end{equation}
has signature $(1,h^{1,1}(X)-1)$ and is negative
definite on
\[
\left\{\gamma\in H^{1,1}(X,\mathbb R):
\int_X\gamma\,\beta_1\cdots\beta_j
\bigl(\Lambda\wedge e^\alpha\bigr)^{[n-j-1]}=0
\right\}.
\]
For $j=0$, the product of the $\beta_i$ is $1$.
\item\label{mainNHR-b}
There exists $\varepsilon_0>0$ such that for every $N\ge1$, nef classes
$\beta_1,\ldots,\beta_N$, and
$0\le\varepsilon<\varepsilon_0$, the degree-$d$
homogeneous polynomial in $(y,\mathbf{x})$, where $\mathbf{x}=(x_1,\ldots,x_N)$,
\begin{equation}\label{intro-master}
\mathcal T_{\alpha,\varepsilon}(y,\mathbf{x})
=y^d\left[\mathcal I\!\left(\alpha+
\frac1y\sum_{i=1}^N x_i\beta_i\right)
+\varepsilon \right]
\end{equation}
is a Lorentzian polynomial.
\end{enumerate}
\end{theorem}

Part~\ref{mainNHR-b} also gives cone-Lorentzianity in the sense of
Br\"and\'en--Leake \cite[Definition~2.1]{NHRBL}: the homogenization
$y^d(\mathcal I(\alpha+\gamma/y)+\varepsilon)$ is Lorentzian on
the product of the positive half-line and the K\"ahler cone.

In the spectral case, $j=d-2$ recovers the mixed Hodge--Riemann
pairing of Dinh--Nguy\^en \cite[Theorem~A]{NHRDN}.
Section~\ref{nhr:section} proves the theorem. Combined with
Corollary~\ref{cor:numerical-solvability}, it gives the Hodge--Riemann
classes illustrated in Section~\ref{nhr:applications}.

\subsection{History}\label{intro:history}
Yau solved the complex Monge--Amp\`ere equation on
compact K\"ahler manifolds in his proof of the Calabi conjecture
\cite{Yau78}, where solvability holds in every K\"ahler class and rests
on a priori estimates up to the third order. For other Hessian equations, solvability in a cohomology class is governed instead by a
subsolution condition, which is in turn expected to be
equivalent to intersection inequalities on subvarieties, in the spirit of
the Demailly--P\u aun numerical characterization of the K\"ahler cone
\cite{DP}. We recall the part of this history closest to the present
paper.

For the complex Hessian equation in a $k$-positive class, Hou--Ma--Wu
proved the second-order estimate \cite{HMW}, and Dinew--Ko{\l}odziej
completed the existence theory on K\"ahler manifolds through a Liouville
theorem that supplies the gradient bound
\cite[Theorems~0.1--0.3]{DK}. Guan \cite{Guan14} and Sz\'ekelyhidi \cite{Szekelyhidi} systematically studied the subsolution conditions in $C^2$ estimates.  Sz\'ekelyhidi \cite{Szekelyhidi} introduced
$\mathcal C$-subsolutions and proved the a priori estimates for a broad
class of equations on Hermitian manifolds, including the Hessian
quotient equations \cite[Corollary~3 and Section~7]{Szekelyhidi}.
Guan developed a closely related weak subsolution condition independently
\cite{Guan14}; these conditions are equivalent on certain convex cones
\cite{Guan24}. For $L^\infty$ estimates, Sui and Sun \cite{SuiSun}
introduced an extended $\mathcal C$-subsolution, which weakens the classical
condition by bounding only the determinant of the relevant positive
perturbations rather than their eigenvalues. We use its general form, the
$\mathcal D$-subsolution of \cite{Linfty}. Second-order estimates for Hessian-type
equations with a quotient structure and nonnegative lower coefficients
were obtained by Guan--Zhang \cite{GuanZhang}, Li--Ren--Wang
\cite{LRW} and Dong \cite{Dong}.

Donaldson \cite{Don99} introduced the $J$-equation through moment maps,
and X.~X.~Chen \cite{Chen00,Chen04} through the Mabuchi energy.
Weinkove \cite{Wei06} proved convergence of the $J$-flow under
a positivity condition, Song--Weinkove \cite{SW08} showed that a cone
condition on the two K\"ahler forms is necessary and sufficient, and
Fang--Lai--Ma \cite{FLM11} extended this to a family of inverse
$\sigma_k$ flows.  See also related work of Guan--Sun \cite{GuanSun} and Sun
\cite{Sun16}  on Hermitian manifolds.

Lejmi--Sz\'ekelyhidi \cite{LS} conjectured the numerical form of the
cone condition; it was proved on toric manifolds by
Collins--Sz\'ekelyhidi \cite{CS}, in uniform form on K\"ahler manifolds
by G.~Chen \cite{GChen}, and in strict form by Song \cite{Song}. The
semistable case, in which the inequalities degenerate, was first
studied on surfaces by Fang--Lai--Song--Weinkove \cite{FLSW14}. For
recent work on the semistable and unstable $J$-equation and on the
related generalized Monge--Amp\`ere and LYZ equations, see
\cite{DMS26,KSD26,Mur26a,Mur26b,XFu26,Liu26a,Liu26b,Liu26c,Liu26d,
FuZhang26,SSD26,FYZZ26,FMJnull}.

The deformed Hermitian--Yang--Mills equation, or Leung--Yau--Zaslow
equation, arose from mirror symmetry \cite{LYZ}. Collins--Jacob--Yau
proved a priori estimates under a subsolution condition and formulated
the numerical conjecture \cite{CJY}. In the supercritical phase, the
criterion was established in uniform form by G.~Chen \cite{GChen} and
in strict form by Chu--Lee--Takahashi \cite{CLT}; the critical phase was
treated by Fu--Yau--Zhang \cite{FYZ}.

For general inverse $\sigma_k$ equations, Lin's theory of
right-Noetherian polynomials gives the top-degree solvability theorem
recalled above \cite{Lin}, with the level-set convexity of
\cite{LinConvex,LinCorr}, and Datar--Pingali proved a numerical
criterion for generalized Monge--Amp\`ere equations on projective
manifolds \cite{DatarPingali}; see also \cite{Mur24}. In the same spirit, our earlier paper
\cite{FMForms} treated top-degree equations with leading term $\omega^n$
and general closed differential-form coefficients, which are not
spectral, and obtained the analytic and numerical solvability criteria
in that setting.

The past few months have seen a burst of activity. Chen--Nie--Xu \cite{CNX} proved a
Nakai--Moishezon type criterion for strongly strictly right-Noetherian
polynomials of degree $n$ on projective manifolds, and a uniform
criterion in every degree. Fu--Li--Zhang--Zhang \cite{FLZZ} removed
projectivity, using mixed intersection numbers with one auxiliary
K\"ahler class and Tosatti's null-supported current
\cite[Proposition~2.1]{Tos} in place of the algebraic input; Corollary~\ref{cor:numerical-solvability} combines their
Theorem~1.2 with Theorem~\ref{main}. Fu--Zhang--Zhang \cite{FZZ26}
obtained a numerical criterion for the $2$-Hessian equation which is
equivalent to solvability in dimension three. On the variational side,
Chen--Ghosh--Nie \cite{CGN} constructed weak geodesics for generalized
Monge--Amp\`ere equations, and Chen--Ghosh--Wang \cite{CGW} proved that
coercivity is equivalent to solvability for right-Noetherian
polynomials, using the quotient concavity of \cite{Ideal} under
proper-position hypotheses.

The algebraic side of this paper belongs to the study of
Hodge--Riemann and Lorentzian structures on cohomology classes.
The mixed Hodge--Riemann relations are due to Dinh--Nguy\^en
\cite{NHRDN}, and also to Cattani \cite{Cattani} by Hodge-theoretic
methods. Xiao \cite{Xiao21a} derived Hodge-index type inequalities, including a mixed version allowing degenerate directions, from G\aa rding's theory of hyperbolic polynomials and the solvability of complex Hessian equations.
Lorentzian polynomials are due to Br\"and\'en--Huh
\cite{NHRBH} and Anari--Liu--Oveis Gharan--Vinzant \cite{AOGV,ALOGV2,ALOGV}. Hu--Xiao introduced Lorentzian classes, whose mixed Hodge--Riemann pairings with K\"ahler classes have Lorentzian signature \cite{HXHodge}. They also proved that the volume polynomial of nef classes is Lorentzian, with reverse Khovanskii--Teissier inequalities, submodularity of the numerical dimension and Pl\"unnecke--Ruzsa type inequalities as consequences \cite{HXMathAnn}; see also \cite{NHRHX}. Lu and Zheng \cite{LuZheng} study Hodge--Riemann polynomials, whose
evaluations at Chern roots of ample vector bundles on smooth projective
varieties yield classes satisfying Hodge--Riemann relations. Theorem~\ref{mainNHR} builds a compatible hierarchy from the derivatives of one nonhomogeneous equation and couples its different degrees in a single Lorentzian polynomial.

\subsection{Method of proof}\label{intro:method}

Starting from an admissible representative $\underline\omega$, we
consider the continuity path
\[
 P_x(A_{\omega_t})=b(t),\qquad
 \omega_t=\underline\omega+t\chi+dd^cv_t,\qquad t\ge0.
\]
The scalar $b(t)$ is determined by the compatibility condition in the
class $[\omega_0]+t[\chi]$, with $b(0)=0$.
The equation is solvable for sufficiently large $t$, and the given
admissible representative remains a subsolution along the path.
It therefore suffices to obtain a priori estimates, uniform on bounded
parameter intervals, to continue the solutions to $t=0$.

For the $L^\infty$ estimate, the subsolution supplies the determinant
bound needed to apply Fang--Ma--Wu \cite{Linfty}. Related estimates
under subsolution conditions were obtained by Sz\'ekelyhidi
\cite{Szekelyhidi}, Sun \cite{Sun17}, and Sui--Sun \cite{SuiSun}.
For the $C^2$ estimate, the concavity theorem for ideal G\aa rding
polynomials gives a fixed concave elliptic formulation of the equation.
This allows us to use Sz\'ekelyhidi's second-order argument
\cite{Szekelyhidi}, following Hou--Ma--Wu \cite{HMW}, to bound the
complex Hessian in terms of the $C^0$ and gradient bounds.

The $C^1$ estimate then follows by the standard blow-up argument
\cite{Szekelyhidi}, using the Liouville theorem of
Dinew--Ko{\l}odziej \cite{DK}. See also \cite{Liouville}. The resulting Hessian bound keeps the
solutions in a compact subset of the elliptic domain. Uniform
ellipticity and concavity give higher-order estimates, which complete
the continuity argument.

For the Hodge--Riemann and Lorentzian conclusions, we use the pointwise
concavity of the mixed derivatives supplied by the ideal G\aa rding
theory \cite{Ideal}. In the spectral case, the quadratic endpoint is
given by the mixed Hodge--Riemann theorem of Dinh--Nguy\^en
\cite{NHRDN}. A linear elliptic correction of cohomology representatives,
as in Xiao \cite{Xiao21a} and Collins \cite{NHRCollins}, allows us to
integrate the pointwise inequalities. The resulting intersection
inequalities give the Hodge--Riemann relations and verify the Hessian
criterion of Br\"and\'en--Huh \cite{NHRBH} for the Lorentzian homogenization.

\subsection{Open problems}\label{intro:open}
The analytic half of the paper uses the spectral structure, while the
algebraic half needs only \ref{H1}--\ref{H2} of
Section~\ref{nhr:section}. Beyond the spectral case, solvability of ideal G\aa rding elliptic equations  is known only for the equations of \cite{FMForms}, with leading term $\omega^n$ and closed negative lower-degree form coefficients under the structural hypotheses imposed there, treated by convexity along multiplicative paths. Open cases include \eqref{general-equation} with form-valued coefficients and other non-spectral G\aa rding elliptic operators of \cite{Linfty}. In general,
relative $L^\infty$ estimates follow from determinantal ellipticity
\cite{Linfty}. The second-order estimate of Subsection~\ref{C2section}
uses the symmetry and concavity of a spectral operator and the
ordering of its eigenvalue derivatives; extending it beyond the setting of
\cite{FMForms} requires a nonspectral replacement. On the numerical side, the
classes $\bigl(\Lambda\wedge e^\alpha\bigr)^{[q]}$ have natural analogues, but the step
from numerical inequalities to an admissible representative is not
known in this generality.

We also ask whether intersection conditions on subvarieties can
characterize solvability and describe the behavior of solutions when
the numerical inequalities degenerate. The models are $J$-stability and its destabilizing
subvarieties for the $J$-equation
\cite{LS,Song,FLSW14,FMJnull,SSD26,Liu26c}, and the boundary cases
of the supercritical deformed Hermitian--Yang--Mills equation
\cite{CJY,Mur26b,FYZZ26}.

\subsection{Organization}

Section~\ref{inputs} collects the algebraic preliminaries, including the
definition of ideal G\aa rding elliptic operators and their basic properties.
Section~\ref{subsolutionsection} constructs the concave spectral operator
and proves the subsolution inequalities. Section~\ref{solvabilitysection}
establishes a priori estimates along the continuity path and proves
Theorem~\ref{main} and Corollary~\ref{cor:numerical-solvability}.
Section~\ref{nhr:section} proves Theorem~\ref{mainNHR} for general solvable ideal
G\aa rding elliptic equations and verifies the standing hypotheses in
the spectral case, then derives the Lefschetz and Hodge-index
consequences. Section~\ref{nhr:applications} gives two spectral examples.

\subsection*{Use of AI tools}
Multiple AI models were used during the preparation of this paper for
mathematical proofreading, the search for examples, and revision of the
exposition. The mathematical ideas and arguments are those of the
authors, who take full responsibility for the content.

\subsection*{Acknowledgements}
The authors thank Jinyang Wu and Jian Xiao for helpful discussions and Gao Chen for his
interest in our work. 

\section{Preliminaries}\label{inputs}

We collect the algebraic results used in later sections. The G\aa rding terminology is that of our theory
\cite{Garding,Ideal}.

For a polynomial $F$ on a real vector space, its directional derivative
at $A$ in the direction $B$ is
\[
 \partial_BF(A)=\left.\frac{d}{dt}F(A+tB)\right|_{t=0}.
\]
Following \cite{Linfty}, we write $\partial F(A)$ and $\partial^2F(A)$
for the differential and Hessian. Thus
\[
 \partial F(A)[B]=\partial_BF(A),\qquad
 \partial^2F(A)[B,C]=\partial_B\partial_CF(A).
\]
Higher derivatives are written in the same way; $\partial_B^k$ means
$k$ successive derivatives in the direction $B$. For polynomials in
coordinates, $\partial_i$ denotes differentiation in the $i$th coordinate.

When $F$ is a polynomial on Hermitian matrices, we also identify its
differential with a Hermitian matrix through the trace pairing:
\[
 \partial F(A)[B]=\operatorname{tr}\bigl(\partial F(A)B\bigr).
\]

For the analytic estimates, we normalize the volume measure by
\[
 dV_\chi=\frac{\chi^n}{n!V},\qquad V=\int_X\frac{\chi^n}{n!},
\]
so that $\int_XdV_\chi=1$. We use the complex Laplacian
$\Delta_\chi=\chi^{i\bar j}\partial_i\partial_{\bar j}$.

\subsection{G\aa rding polynomials and polarization}

We recall the two recursive definitions from
\cite[Definitions~1.1 and 2.1]{Ideal}.

\begin{definition}[G\aa rding and ideal G\aa rding polynomials]
\label{def:garding-ideal}
For a real nonzero polynomial $G$ in $n$ variables, its \emph{G\aa rding component}
is the connected component $\C_G$ of $\{G>0\}$ satisfying
\[
 \C_G+\mathbb R^n_{>0}\subset\C_G,
\]
when such a component exists. Here $\mathbb R^n_{>0}=\{(x_1,\cdots,x_n)|\,x_i>0\}$.

All nonnegative constant polynomials, including the zero polynomial,
are G\aa rding and ideal G\aa rding. By convention their G\aa rding
components are $\mathbb R^n$; in particular, $\C_0=\mathbb R^n$.
For nonconstant polynomials, define recursively:
\begin{enumerate}[label=\textup{(\roman*)}]
\item $G$ is \emph{G\aa rding} if $\C_G$ exists, every 
partial derivative $\partial_iG$ is G\aa rding, and
\[
 \C_G\subset
 \bigcap_{i}\C_{\partial_iG}.
\]
\item $G$ is \emph{ideal G\aa rding} if it is G\aa rding,
$\C_G$ is convex, and every partial derivative
$\partial_iG$ is ideal G\aa rding.
\end{enumerate}
\end{definition}

For multi-affine polynomials (degree at most one in
each variable), the G\aa rding condition is equivalent to the existence
of $\C_G$, and the ideal condition is equivalent to its existence and
convexity \cite[Remark~1.2]{Ideal}.
For a G\aa rding polynomial $G$ of degree at least two, its
\emph{first derived component} is
$\C_G^{(1)}=\bigcap_{i=1}^n\C_{\partial_iG}$. Then
\cite[Propositions~5.5 and 5.7]{Garding} give
\begin{equation}\label{componentidentity}
 \C_G=\C_G^{(1)}\cap\{G>0\},\qquad
 G\le0\quad\text{on }\partial\C_G^{(1)}.
\end{equation}

The following polarization theorem is \cite[Theorem~8.2]{Ideal}.

\begin{theorem}\label{polarization}
For $G(\lambda)=\sum_{j=0}^d b_j\sigma_j(\lambda)$ with $b_d>0$ and $d\le n$,
the following are equivalent: $s\mapsto G(s\one)$ is right-Noetherian;
$s\mapsto G(s\one)$ is a univariate G\aa rding polynomial; $G$ is
G\aa rding; and $G$ is ideal G\aa rding.
\end{theorem}

Theorem~\ref{polarization} follows from the univariate identification
in \cite[Section~9.2]{Garding} and the polarization result in
\cite[Theorems~8.1--8.2]{Ideal}; it applies to every $d\le n$,
including the top-degree case $d=n$ treated by Lin
\cite{LinConvex,LinCorr}. See also \cite[Remark~9.14]{Garding}.

The component and concavity properties used below are collected from
\cite{Ideal}: part (i) follows from the definitions and
\cite[Theorem~11.1(2),(4)]{Ideal}, part (ii) is \cite[Theorem~12.1]{Ideal}, and
part (iii) is \cite[Theorem~13.3(1)]{Ideal}.

\begin{proposition}\label{idealinputs}
Let $G$ be a nonzero ideal G\aa rding polynomial.
\begin{enumerate}[label=\textup{(\roman*)}]
\item Its G\aa rding component is convex and invariant under
nonnegative coordinate translations. Every nonzero iterated derivative
in nonnegative directions is ideal G\aa rding and is positive on $\C_G$;
its G\aa rding component contains $\C_G$.
\item If $F$ is such a derivative (including $G$ itself) of degree $m\ge1$, then $F^{1/m}$
and $\log F$ are concave on $\C_F$.
\item If $v\ge0$ and $\partial_vG\not\equiv0$, then $G/\partial_vG$ is concave on
$\C_{\partial_vG}$.
\end{enumerate}
\end{proposition}

\subsection{Ideal G\aa rding elliptic operators}\label{subsec:ige}

Let $\mathcal Q(E)$ be the real vector space of Hermitian forms on an
$n$-dimensional complex vector space $E$, with semipositive cone
$\mathcal Q(E)^+=\{D\in\mathcal Q(E):D\ge0\}$.
We use the definition of G\aa rding elliptic operators from
\cite[Definitions~4.4, 4.5 and 4.12]{Linfty}.

\begin{definition}[G\aa rding and ideal G\aa rding elliptic operators]\label{def:ige}
Nonnegative constants, including zero, are G\aa rding elliptic and
ideal, with component $\mathcal Q(E)$.
A nonconstant real polynomial $G$ on $\mathcal Q(E)$ is a
\emph{G\aa rding elliptic operator} if it has a component $\C_G$ of
$\{G>0\}$ satisfying
\[
 \C_G+\mathcal Q(E)^+\subset\C_G,
\]
and every nonzero derivative $\partial_DG$, $D\ge0$, is G\aa rding
elliptic with $\C_G\subset\C_{\partial_DG}$.
We call $\C_G$ its \emph{G\aa rding component}.

Such an operator is \emph{ideal} if the G\aa rding components of all
its nonzero iterated derivatives in semipositive directions are convex,
including $G$ itself.
\end{definition}

\begin{lemma}\label{ige:basic}
Let $G$ be a G\aa rding elliptic operator on $\mathcal Q(E)$.
\begin{enumerate}[label=\textup{(\roman*)}]
\item For $A\in\C_G$ and $D_1,\ldots,D_N\ge0$ with $\sum_iD_i>0$, the
\emph{positive scalarization}
\[
 f(\mathbf{x})=G\left(A+\sum_ix_iD_i\right),\qquad \mathbf{x}\in\mathbb R^N,
\]
is G\aa rding, with component
\[
 \C_f=\left\{\mathbf{x}\in\mathbb R^N:A+\sum_ix_iD_i\in\C_G\right\}.
\]
\item $G$ is ideal if and only if all its positive scalarizations are ideal
G\aa rding polynomials.
\item If $G$ is ideal and $F=\partial_{P_1}\cdots\partial_{P_k}G\not\equiv0$,
where $P_i\ge0$ and $k\ge0$, then $F$ is ideal G\aa rding elliptic with
$\C_G\subset\C_F$. For $m=\deg F\ge1$, both $F^{1/m}$ and $\log F$
are concave on $\C_F$.
\end{enumerate}
\end{lemma}

\begin{proof}
Part (i) is \cite[Theorem~4.7]{Linfty}.

For (ii), write $L(\mathbf{x})=A+\sum_ix_iD_i$. Directional differentiation
commutes with this affine substitution:
\[
 \partial_{x_i}(G\circ L)=(\partial_{D_i}G)\circ L.
\]
For each nonzero iterated derivative $F$ in these directions, part (i) gives
\[
 \C_{F\circ L}=L^{-1}(\C_F).
\]
If $G$ is ideal, all these inverse images are convex, so its
scalarizations are ideal. Conversely, extend the directions of any
iterated derivative $F$ to a semipositive family spanning $\mathcal Q(E)$.
Then $L$ is onto, and convexity of $\C_{F\circ L}$ implies convexity of
$\C_F=L(\C_{F\circ L})$.

For (iii), the definition shows that $F$ is ideal and
$\C_G\subset\C_F$. Apply Proposition~\ref{idealinputs} to a positive
scalarization of $F$ in a spanning family. Part (i) and an affine right
inverse of $L$ transfer the concavity inequalities to $\C_F$.
\end{proof}

\begin{lemma}\label{ige:criterion}
Let $G$ be a nonconstant polynomial on $\mathcal Q(E)$, and suppose
$\{G>0\}$ has a component $\C_G$ invariant under semipositive increments. If
\[
 \mathbf{x}\longmapsto G\left(A_0+\sum_ix_iD_i\right)
\]
is G\aa rding for every Hermitian $A_0$ and every semipositive family
with $\sum_iD_i>0$, then $G$ is G\aa rding elliptic.
\end{lemma}

\begin{proof}
We argue by induction on the degree of $G$. If $G$ is linear, then the existence of $\C_G$ implies that $G$ is G\aa rding elliptic. 

Assume $\deg G\geq2$. Let $P,Q\ge0$ with $F=\partial_PG\not\equiv0$. Choose a semipositive
spanning family with $D_1=P$ and $D_2=Q$, and put
$L(\mathbf{x})=\sum_ix_iD_i$ and $g=G\circ L$.
 Then $L^{-1}(\C_G)$ is a connected component of $\{g>0\}$ and is invariant under nonnegative coordinate translations. Hence $\C_g=L^{-1}(\C_G)$. The inclusion $\C_g\subset\C_{\partial_{x_1}g}$ then gives
\[
 \C_g=L^{-1}(\C_G),\qquad
 \C_{\partial_{x_1}g}=L^{-1}(U),\qquad \C_G\subset U,
\]
where $U$ is the component of $\{F>0\}$ containing $\C_G$.
Since $\partial_{x_1}\partial_{x_2}g\geq0$ in $L^{-1}(U)$, $U+Q\subset U$. Since $Q$ is
arbitrary, $U$ is invariant under all semipositive increments.

Next let $A_0$ be Hermitian and $D_1',\ldots,D_M'\ge0$ with
$\sum_jD_j'>0$. 
\[\tilde{g}(y,\mathbf{x}):=G(A_0+yP+\sum_jx_jD_j')\] is G\aa rding by the assumption. Let $T>0$. For large $T$, the point $(0,T\one)$ lies in $\C_{\tilde{g}}\subset\C_{\partial_y\tilde{g}}$ and $\tilde{g} (y,\mathbf x+T\one)$ has non-negative coefficients.
\[
F\left(A_0+\sum_j(x_j+T)D_j'\right)=\partial_y\tilde{g}(0,\mathbf{x}+T\one).
\]
Since the restriction map is a weak positive linear transform with a real stable symbol, by \cite[Theorem~10.3]{Garding}, it preserves G\aa rding polynomials with nonnegative coefficients.
In particular, $\partial_y\tilde{g}(0,\mathbf{x}+T\one)$ is G\aa rding and by translation invariance, so is $\partial_y\tilde{g}(0,\mathbf{x})$. 
Induction on the degree
therefore shows that $F$ is G\aa rding elliptic with
$\C_F=U\supset\C_G$, proving the assertion.
\end{proof}

\subsection{Spectral operators}\label{subsec:spectral} We show that scalar symmetric ideal G\aa rding polynomials lift to spectral ideal G\aa rding elliptic operators. 
\begin{lemma}\label{nhr:scalarization}
Let $q=\sum_{j=0}^d b_j\sigma_j$ be an ideal G\aa rding polynomial on
$\mathbb R^n$, with $b_d>0$ and $1\le d\le n$.
Then $F(A)=q(\lambda(A))$ is an ideal G\aa rding
elliptic operator, with component
\[
 \C_F=\{A:\lambda(A)\in\C_q\}.
\]
For every Hermitian $A_0$ and every semipositive family
$D_1,\ldots,D_N$ with $\sum_iD_i>0$, the polynomial
\[
 f(\mathbf{x})=F\left(A_0+\sum_i x_iD_i\right),
 \qquad \mathbf{x}\in\mathbb R^N,
\]
is ideal G\aa rding and has degree $\min\{d,\operatorname{rank}D_i\}$
in each variable $x_i$.
\end{lemma}

\begin{proof}
The component $\C_q$ is symmetric. Every $A\in\C_F$ can be joined by
unitary conjugation to $\operatorname{diag}\lambda(A)$, so connectedness
of $\C_q$ gives connectedness of $\C_F$. Continuity of the ordered
eigenvalues makes $\C_F$ open and closed in $\{F>0\}$, hence a connected
component. Weyl's inequalities give invariance under semipositive increments.

Fix $A_0$ and the $D_i$. Choose $a\one\in\C_q$. The translate
$g(\boldsymbol{\lambda})=q(\boldsymbol{\lambda}+a\one)$ is ideal
with nonnegative coefficients \cite[Lemma~3.1]{Ideal}. For sufficiently
large $\tau$, the matrix polynomial
\[
 M(\mathbf{x})=A_0-aI+\sum_i(x_i+\tau)D_i
\]
has positive definite constant term. Define a linear map $S$ on
multi-affine monomials by
\[
 S\left(\prod_{j\in J}\lambda_j\right)=\det M(\mathbf{x})_J,
 \qquad J\subseteq\{1,\ldots,n\},
\]
where $M_J$ denotes the principal submatrix indexed by $J$.
Its symbol is
\[
 S\left(\prod_{j=1}^n(\lambda_j+u_j)\right)
 =\det\left(M(\mathbf{x})+\operatorname{diag}(\mathbf{u})\right).
\]
The mixed-discriminant expansion gives nonnegative coefficients.
When all $x_i,u_j$ have positive imaginary parts, the matrix has positive
definite imaginary part, so the symbol is real stable and hence ideal
G\aa rding \cite[Proposition~4.3 and Theorem~1.4(1)]{Ideal}.
The symbol criterion \cite[Theorem~10.2]{Ideal} therefore makes $Sg$
ideal G\aa rding. Since
\[
 (Sg)(\mathbf{x})
 =q\bigl(\lambda(M(\mathbf{x}))+a\one\bigr)
 =F\left(A_0+\sum_i(x_i+\tau)D_i\right),
\]
translation invariance proves the assertion about $f$.
Lemmas~\ref{ige:criterion} and \ref{ige:basic}\textup{(ii)} then show
that $F$ is ideal G\aa rding elliptic.

For the degree, put $r=\operatorname{rank}D_i$. Principal-minor expansion
gives $\deg_{x_i}f\le\min\{d,r\}$. If $r\ge d$, the coefficient of
$x_i^d$ is $b_d\sigma_d(D_i)>0$. If $r<d$, set $x_j=\tau$ for $j\ne i$.
The coefficient of $x_i^r$ is
\[
 \sigma_r(D_i)\sum_{\ell=r}^d b_\ell\,
 \sigma_{\ell-r}\left(
 \left.\left(A_0+\tau\sum_{j\ne i}D_j\right)\right|_{\ker D_i}\right),
\]
where the restriction denotes compression to $\ker D_i$.
The compression of $\sum_{j\ne i}D_j$ is positive definite there.
Thus the term $\ell=d$ has positive leading coefficient in $\tau$ of
degree $d-r$, while the other terms have smaller degree. The displayed
coefficient is therefore nonzero for large $\tau$.
\end{proof}

\section{A concave spectral operator and subsolutions}\label{subsolutionsection}
We work with the spectral polynomial of \eqref{polynomial} under
\eqref{strict}. We construct a fixed concave and elliptic operator $F$ and,
for each $b\ge0$, a positive function $h$ such that, for $A\in\mathcal A$,
\[
 P_x(A)=b\quad\Longleftrightarrow\quad F(A)=h(x).
\]
We also show that every $\underline A\in\mathcal A$ gives a quantitative
subsolution of this equation, in the sense of
Proposition~\ref{raylemma}.

Recall that
\[
 P_x(\lambda)=\widetilde P(\lambda)+a_0(x),\qquad
 \widetilde P(\lambda)=\sum_{j=1}^d a_j\sigma_j(\lambda),
 \qquad
 \mathcal A=\bigcap_{i=1}^n\C_{P_i},\quad P_i=\partial_i\widetilde P.
\]
Its diagonal derivative is
\begin{equation}\label{diagonal-derivative}
 \partial_{\one}\widetilde P
 =\sum_{j=1}^d(n-j+1)a_j\sigma_{j-1}>0
 \qquad\text{on }\mathcal A.
\end{equation}
Let $\rho_1$ be the largest root of $s\mapsto(\partial_{\one}\widetilde P)(s\one)$.
Since $s\mapsto P_x(s\one)$ is increasing for $s>\rho_1$, the strict first
root gap is equivalent to
\begin{equation}\label{scalar-gap}
 \delta(x):=-P_x(\rho_1\one)>0,\qquad
 \delta_{\min}:=\min_X\delta>0.
\end{equation}

\begin{lemma}\label{shiftlemma}
For every $x\in X$ and $b>-\delta(x)$, the polynomial $P_x-b$ is
ideal G\aa rding with first derived component $\mathcal A$, and
\begin{equation}\label{admissible-levels}
 P_x\le-\delta(x)\quad\text{on }\partial\mathcal A,
 \qquad \overline{\C_{P_x-b}}\subset\mathcal A.
\end{equation}
\end{lemma}

\begin{proof}
The diagonal polynomial $s\mapsto P_x(s\one)-b$ is negative at
$\rho_1$ and strictly increasing above it. It therefore has exactly one
root above $\rho_1$, and its positive-order derivatives retain the
monotone root sequence of \eqref{strict}. Theorem~\ref{polarization}
makes $P_x-b$ ideal G\aa rding, with first derived component $\mathcal A$.

By \eqref{componentidentity}, $P_x-b\le0$ on
$\partial\mathcal A$. Letting $b\downarrow-\delta(x)$ gives the first
inequality in \eqref{admissible-levels}. Moreover,
\[
 \overline{\C_{P_x-b}}
 \subset\overline{\mathcal A}\cap\{P_x\ge b\}.
\]
The right-hand side cannot meet $\partial\mathcal A$, where
$P_x\le-\delta(x)<b$. This proves the second assertion.
\end{proof}

Choose
\begin{equation}\label{root-shift}
 \delta_1=\frac{\delta_{\min}}2,\qquad
 c=\max_Xa_0+\delta_1,\qquad
 G=\widetilde P+c.
\end{equation}
Since $G(\rho_1\one)=-\delta_1<0$, Theorem~\ref{polarization} makes
$G$ ideal G\aa rding. Lemma~\ref{shiftlemma} gives
$G\le-\delta_1$ on $\partial\mathcal A$. Put
\begin{equation}\label{eq:root-shift-F}
 F=G^{1/d}\quad\text{on }\C_G=\mathcal A\cap\{G>0\}.
\end{equation}
For $A\in\mathcal A$ and $b\ge0$, the equation $P_x(A)=b$ is equivalent to
\[
 F(A)=h(x),\qquad h(x)=(b+c-a_0(x))^{1/d}\ge\delta_1^{1/d}.
\]

\begin{proposition}\label{rootoperator}
The spectral operator $F$ is concave and strictly elliptic on $\C_G$,
and
\[
 \operatorname{tr}(\partial F(A))\ge\left(\binom nd a_d\right)^{1/d}
 \qquad(A\in\C_G).
\]
\end{proposition}

\begin{proof}
Since $G$ is ideal, Lemmas~\ref{nhr:scalarization} and
\ref{ige:basic}\textup{(iii)} give concavity of $F=G^{1/d}$ on $\C_G$.
The identity
\[
 \partial F(A)=\frac{\partial\widetilde P(A)}{dG(A)^{(d-1)/d}}
\]
gives strict ellipticity. For sufficiently large $r$, concavity gives
\[
 F(rI)\le F(A)+\partial F(A)[rI-A].
\]
Dividing by $r$ and letting $r\to\infty$ proves the trace bound.
\end{proof}

\begin{proposition}\label{raylemma}
Let $K\subset\mathcal A$ be compact and $b_0\ge0$.
There are $\varrho>0$, $L\ge1$ and $\kappa>0$, uniform for
$\underline A\in K$, $x\in X$ and $0\le b\le b_0$, such that
$\underline A-\varrho I\in\mathcal A$ and every Hermitian matrix $A$
satisfying
\[
 P_x(A)=b,\qquad A-\underline A+\varrho I\ge0
\]
obeys
\begin{equation}\label{quantitative-subsolution}
 \|A-\underline A+\varrho I\|<L.
\end{equation}
Moreover, every $A\in\mathcal A$ with $P_x(A)=b$ satisfies either
\begin{equation}\label{alternatives}
 \partial F(A)[\underline A-A]\ge\kappa\operatorname{tr}(\partial F(A)),
 \qquad\text{or}\qquad
 \partial F(A)\ge\kappa\operatorname{tr}(\partial F(A))I.
\end{equation}
The norm is induced by the trace pairing.
\end{proposition}

\begin{proof}
Choose $\varrho>0$ so that $U=\underline A-\varrho I$, for
$\underline A\in K$, stays in a compact subset of $\mathcal A$.
Ellipticity gives $\partial\widetilde P(U)[\zeta\zeta^*]\ge\eta>0$ uniformly for
unit vectors $\zeta$. Since $P_x$ is affine under rank-one increments,
\[
 P_x(U+r\zeta\zeta^*)-b
 =P_x(U)-b+r\,\partial\widetilde P(U)[\zeta\zeta^*].
\]
Choose $L\ge1$ so that the right-hand side is positive for
$r\ge L/\sqrt n$, uniformly for $0\le b\le b_0$.
If $N=A-U\ge0$ and $\|N\|\ge L$, an eigenvector for its largest
eigenvalue gives $N\ge(L/\sqrt n)\zeta\zeta^*$. Thus
\[
 b=P_x(U+N)\ge P_x\bigl(U+(L/\sqrt n)\zeta\zeta^*\bigr)>b,
\]
a contradiction. This proves \eqref{quantitative-subsolution}.

The alternatives are the matrix subsolution dichotomy of
Sz\'ekelyhidi \cite[Proposition~6]{Szekelyhidi}, following Guan
\cite{Guan14}. Its argument uses convexity and ellipticity of the level
sets and the boundedness just proved, so it applies to $\C_G$.
For bounded $A$, the inequalities $G(A)\ge\delta_1$ and
$G\le-\delta_1$ on $\partial\mathcal A$ place the solution matrices in
a compact subset of $\C_G$. Strict ellipticity then gives the second
alternative after decreasing $\kappa$.
\end{proof}

\section{A priori estimates and solvability}\label{solvabilitysection}
We prove Theorem~\ref{main} by the continuity method. Fix an admissible representative
\[
 \underline\omega=\omega_0+dd^c\underline u,\qquad
 \underline A=\chi^{-1}\underline\omega\in\mathcal A.
\]
In the class
$[\omega_0]+t[\chi]$ with $t\ge0$, compatibility determines
\begin{equation}\label{bt}
 b(t)=\frac{1}{V}\int_X \Lambda \wedge e^{\omega_0+t\chi}=\int_XP_x(\underline A+tI)\,dV_\chi.
\end{equation}
For $t\in[0,\infty)$, the continuity path is
\begin{equation}\label{path}
 \begin{aligned}
 P_x(A_t)&=b(t),\qquad A_t=\chi^{-1}\omega_t\in\mathcal A,\\
 \omega_t&=\underline\omega_t+dd^cv_t,\qquad
 \underline\omega_t=\underline\omega+t\chi.
 \end{aligned}
\end{equation}
Write $$\underline A_t=\chi^{-1}\underline\omega_t=\underline A+tI.$$
Compatibility gives $b(0)=0$. Since $\underline A+tI\in\mathcal A$,
\begin{equation}\label{bprime}
b'(t)=\int_X\partial_{\one}\widetilde P(\underline A+tI)\,dV_\chi>0.
\end{equation}
In particular, $b(t)\ge0$.

With $F=G^{1/d}$ as in \eqref{eq:root-shift-F}, this is
\begin{equation}\label{root-path}
 F(A_t)=h_t(x),\qquad
 h_t(x)=\bigl(b(t)+c-a_0(x)\bigr)^{1/d}.
\end{equation}
Since $b(t)+c-a_0(x)\ge\delta_1$, the solution matrices lie in
$\C_G$. Fix $T<\infty$. For $0\le t\le T$,
\begin{equation}\label{rhs-bounds}
 h_t\ge\delta_1^{1/d},\qquad
 \|h_t\|_{C^2}\le C_T.
\end{equation}
The matrices $\underline A_t(x)$ range in a compact subset of
$\mathcal A$. Proposition~\ref{raylemma}, with $b_0=b(T)$, therefore
gives uniform constants $\varrho,L,\kappa$.

\subsection{\texorpdfstring{The $L^\infty$-estimate}{The L-infinity estimate}}
The $L^\infty$ estimate follows from \cite[Theorem~1.1]{Linfty}.
We verify its hypotheses using a trace lower bound on $\mathcal A$
and a determinant bound for semipositive increments.

\begin{proposition}\label{C0prop}
Fix $T$ and let $0\leq t\leq T$. Every solution of \eqref{path} normalized by $\sup_Xv_t=0$ satisfies
\[
 \|v_t\|_{L^\infty}\le C_T.
\]
\end{proposition}

\begin{proof}
Every nonzero iterated coordinate derivative of $\widetilde P$ of positive order is
positive on $\mathcal A$. Hence
\begin{align}\label{tracebound}
 0<\partial_{\one}^{d-1}\widetilde P
 &=\frac{(n-1)!}{(n-d)!}a_d\sigma_1
   +\frac{n!}{(n-d+1)!}a_{d-1},\notag\\
 \operatorname{tr}A
 &>-\frac{n a_{d-1}}{(n-d+1)a_d}
 \qquad(A\in\mathcal A).
\end{align}
Since $\underline A_t$ is bounded, this gives
\[
 \Delta_\chi v_t=\operatorname{tr}A_t-
     \operatorname{tr}\underline A_t\ge-C_T.
\]
The normalization and Green's formula then yield
$\int_X(-v_t)\,dV_\chi\le C_T$.

For the spectral polynomial $G$, \cite[Proposition~4.16]{Linfty}
shows that $(G,\mathcal A)$ is an admissible pair satisfying the quasi-subharmonicity condition, and gives
\begin{equation}\label{increment}
 G(B+Q)-G(B)\ge a_d\binom nd\det\nolimits_\chi(Q)^{d/n}
 \qquad(B\in\mathcal A,\ Q\ge0).
\end{equation}
Apply this with $B=\underline A_t-\varrho I\in\mathcal A$. Both
$G(B)$ and $h_t^d$ are uniformly bounded, so
\[
 Q\ge0,\quad G(B+Q)\le h_t^d
 \quad\Longrightarrow\quad \det\nolimits_\chi Q\le C_T.
\]
Theorem~1.1 of \cite{Linfty} now applies with auxiliary form
$\varrho\chi$ and constant determinant density, proving the estimate.
\end{proof}

\subsection{The second-order estimate}\label{C2section}
Given the $L^\infty$ bound, we use Sz\'ekelyhidi's
second-order estimate \cite[Proposition~13]{Szekelyhidi} to control
the complex Hessian in terms of the gradient.

\begin{proposition}\label{C2prop}
For every finite $T$, each smooth solution of \eqref{path} satisfies
\begin{equation}\label{C2bound}
 \|dd^cv_t\|_{L^\infty}
 \le C_T\bigl(1+\|\partial v_t\|_{L^\infty}^2\bigr).
\end{equation}
\end{proposition}

\begin{proof}
For \eqref{root-path}, Proposition~\ref{rootoperator} gives
concavity, ellipticity and
\[
 \operatorname{tr}\partial F(A)\ge\left(\binom nd a_d\right)^{1/d}>0.
\]
Moreover, $F$ vanishes on $\partial\C_G$, while $h_t\ge\delta_1^{1/d}$.
Proposition~\ref{raylemma} gives the subsolution bound
\eqref{quantitative-subsolution} and the alternatives
\eqref{alternatives}. Taking matrices diagonal in an eigenbasis of
$\underline A_t$ gives the eigenvalue subsolution condition used in
that proof. The bounds for $h_t$ and $\operatorname{osc}_Xv_t$ follow
from \eqref{rhs-bounds} and Proposition~\ref{C0prop}.

The domain $\C_G$ need not be a cone. In the second-order calculation,
the cone condition is used to control the negative eigenvalues by the
largest one. Here, if $\lambda_1\ge\cdots\ge\lambda_n$ are the
eigenvalues of $A_t$, \eqref{tracebound} gives
\[
 \lambda_n\ge-C-(n-1)\lambda_1,\qquad
 |\lambda_i|\le C\lambda_1\quad\text{when }\lambda_1\ge1.
\]
This replaces the trace positivity in Sz\'ekelyhidi's calculation
of the largest eigenvalue; the resulting inequalities differ only
by fixed constants. His maximum-principle argument therefore yields
\[
 \sup_X\lambda_{\max}(A_t)
 \le C_T\bigl(1+\|\partial v_t\|_{L^\infty}^2\bigr).
\]
Using \eqref{tracebound} once more and subtracting
$\underline\omega_t$ proves \eqref{C2bound}.
\end{proof}

\subsection{Gradient and higher-order estimates}
To close \eqref{C2bound}, we bound the gradient by a blow-up argument.
We then verify uniform ellipticity and obtain higher regularity.

\begin{proposition}\label{smoothprop}
For every finite $T$ and every integer $m\ge0$, the solutions of
\eqref{path} with $0\le t\le T$, normalized by either supremum zero
or integral zero, satisfy
\[
 \|v_t\|_{C^m}\le C_{m,T}.
\]
\end{proposition}

\begin{proof}
For the gradient bound, we use the standard blow-up argument of
Sz\'ekelyhidi \cite[Section~6]{Szekelyhidi}. If the bound fails, choose normalized
solutions $v_j$ and rescale at their gradient maximum points, in
holomorphic charts unitary at their centers:
\[
 r_j=\sup_X|\partial v_j|_\chi\longrightarrow\infty,
 \qquad w_j(z)=v_j(z/r_j).
\]
Propositions~\ref{C0prop} and~\ref{C2prop} give a subsequence with
\[
 w_j\longrightarrow w\text{ in }C^{1,\alpha}_{\mathrm{loc}}(\mathbb C^n),
 \qquad |\partial w(0)|=1,
\]
where $w$ is bounded and globally Lipschitz. After rescaling, the
background and lower-degree terms vanish, leaving $a_d\sigma_d$.
The positive diagonal derivative inequalities on the admissible domain
select the $\Gamma_d$ branch and give
\begin{equation}\label{limitPDE}
 \lambda(dd^cw)\in\partial\Gamma_d,\qquad
 \Gamma_d=\{\sigma_1>0,\ldots,\sigma_d>0\},
\end{equation} in the viscosity sense.
For $d<n$, the Dinew--Ko{\l}odziej Liouville theorem
\cite[Theorem~0.1]{DK} makes $w$ constant. For $d=n$, bounded
plurisubharmonic functions on $\mathbb C^n$ are constant.
Both contradict $|\partial w(0)|=1$.

The gradient bound and \eqref{C2bound} give $\|A_t\|\le C_T$.
Since $G(A_t)\ge\delta_1$ and $G\le-\delta_1$ on
$\partial\mathcal A$, the solution matrices lie in
\[
 \{A\in\overline{\mathcal A}:\|A\|\le C_T,\ G(A)\ge\delta_1\}
 \Subset\C_G.
\]
Strict ellipticity of $F$ therefore gives uniform ellipticity near
all solution matrices. Using the complex Evans--Krylov theorem in \cite[Theorem~3.4 and Section~4]{TWWY},
we obtain a uniform $C^{2,\alpha}$ bound from \eqref{rhs-bounds}. Standard Schauder estimates give all higher bounds. 
\end{proof}
\begin{remark}
In our previous joint work with Wu \cite{Liouville}, we obtained a full geometric characterization of Liouville rigidity for degenerate real and complex Hessian equations. In particular, the result applies to all spectral cases and provides an alternative to the Liouville theorem used above.
\end{remark}

\subsection{The continuity method}\label{continuitysection}
The estimates above give closedness on finite parameter intervals.
To start the path, we solve a perturbation of a Poisson equation
for large $t$.

\begin{proposition}\label{startprop}
Equation \eqref{path} has a smooth solution for every sufficiently
large $t$. Those solutions satisfy $\omega_t>0$.
\end{proposition}

\begin{proof} For $k\geq0,\alpha\in(0,1)$, let \[
 C^{k,\alpha}_0(X):=\left\{v\in C^{k,\alpha}(X,\mathbb R)\,|\ \int_XvdV_\chi=0\right\}.\]
Set $\tau=t^{-1}$. For $v\in C^{2,\alpha}_0(X)$, write $\underline{\omega}_v:=\underline \omega+\ddc v$ and define
\[
 \mathcal G(\tau,v)=\tau^{d-1}
 \bigl[P_x(\tau^{-1}I+A_{\underline{\omega}_v})-b(\tau^{-1})\bigr].
\]
Since $$\frac{(\Lambda\wedge e^{\underline{\omega}_v+t\chi})^{[n]}}{\chi^n/n!}= 
\binom nd a_{d}t^d+O(t^{d-1}),\qquad b(t)=\binom nd  a_dt^d+O(t^{d-1}),$$
the expansion 
shows that the leading $t^d$ terms in $P_x(tI+A_{\underline{\omega}_v})-b(t)$ cancel. Thus $\mathcal G$ extends smoothly to $\tau=0$, with
\begin{equation}\label{Poissonlimit}
 \mathcal G(0,v)=\binom{n-1}{d-1}a_d
 \left(\Delta_\chi v+\operatorname{tr}_\chi\underline\omega
       -\int_X\operatorname{tr}_\chi\underline\omega\,dV_\chi\right).
\end{equation}
The equation $\mathcal G(0,v)=0$ has a unique smooth
solution $v_*$ with zero average. For every $\tau$ and $v$, the map has integral zero:
\[
 \int_X\mathcal G(\tau,v)\,dV_\chi=0.
\]
At $(0,v_*)$, the derivative in $v$ is the isomorphism
\[
 \binom{n-1}{d-1}a_d\Delta_\chi:
 C^{2,\alpha}_0(X)\longrightarrow C^{0,\alpha}_0(X),
\]
where the subscript denotes integral zero. The implicit-function
theorem gives $v_\tau$ for small positive $\tau$, uniformly bounded in
$C^{2,\alpha}$. For $t=\tau^{-1}$ sufficiently large,
\[
 \omega_t=t\chi+\underline\omega+dd^cv_\tau>0.
\]
Its eigenvalues are $t+O(1)$ and therefore belong to $\mathcal A$. Elliptic regularity makes $\omega_t$ smooth.
\end{proof}

\begin{proof}[Proof of Theorem~\ref{main}]
Assume \textup{(i)}. Choose $T$ large enough for
Proposition~\ref{startprop}, and let $\mathcal S\subset[0,T]$ be the
set of parameters for which \eqref{path} has a smooth admissible
solution, normalized by integral zero. Then $T\in\mathcal S$.

For openness, the map
\[
 v\longmapsto P_x(A_{\underline\omega_t+dd^cv})-b(t)
\]
takes values in mean-zero functions by \eqref{bt}.
Its linearization at a solution is
\begin{equation}\label{linearization}
 L_tw\,\frac{\chi^n}{n!}=dd^cw\wedge\Xi_t,
 \qquad \Xi_t=(\Lambda\wedge e^{\omega_t})^{[n-1]}.
\end{equation}
The form $\Xi_t$ is closed and strictly positive, since $P_i>0$ on
$\mathcal A$. For real smooth $w,z$, integration by parts gives
\[
 \begin{aligned}
 \int_XwL_tz\,\frac{\chi^n}{n!}
 &=-\operatorname{Re}\int_X\sqrt{-1}\,
         \partial w\wedge\bar\partial z\wedge\Xi_t\\
 &=\int_XzL_tw\,\frac{\chi^n}{n!}.
 \end{aligned}
\]
The kernel consists of constants by ellipticity and the maximum
principle. Self-adjointness and Fredholm theory make $L_t$ an
isomorphism between the mean-zero H\"older spaces, so the
implicit-function theorem proves openness.

For closedness, Proposition~\ref{smoothprop} gives uniform bounds on
$[0,T]$ and keeps the solution matrices in a compact subset of
$\mathcal A$. If $t_j\in\mathcal S$ converges, a subsequence of the
normalized solutions converges smoothly. Its limit solves the
endpoint equation and remains admissible. Hence $\mathcal S$ is closed,
and connectedness gives $\mathcal S=[0,T]$.
At $t=0$, $b(0)=0$ and $[\omega_t]=[\omega_0]$, so $\omega_t$ satisfies \eqref{target}, proving \textup{(ii)}.

Conversely, a solution in \textup{(ii)} lies in
$\overline{\C_{P_x}}\subset\mathcal A$ by \eqref{admissible-levels}, and
therefore supplies \textup{(i)}.
For uniqueness, let $\omega_u,\omega_v$ be two solutions. Both belong
to $\C_G$ and satisfy $F(A_u)=F(A_v)=h_0(x)$. Concavity gives
\[
 0=F(A_u)-F(A_v)
 \le\partial F(A_v)[\chi^{-1}dd^c(u-v)].
\]
The right-hand side is a smooth strictly elliptic linear operator
applied to $u-v$. The maximum principle on the connected compact
manifold forces $u-v$ to be constant.
\end{proof}

\subsection{The numerical criterion}\label{numericalcriterion}
The numerical inequalities yield an admissible representative by Fu--Li--Zhang--Zhang
\cite{FLZZ}, to which
Theorem~\ref{main} applies.

\begin{proof}[Proof of Corollary~\ref{cor:numerical-solvability}]
Assume \textup{(i)}. Replace $a_0(x)$ by its average
$\bar a_0=\int_Xa_0\,dV_\chi$ and put
$\bar P=\widetilde P+\bar a_0$. This preserves compatibility and the
mixed intersection classes of degree $q<n$. Moreover,
\[
 -\bar P(\rho_1\one)=-\int_XP_x(\rho_1 \one)\,dV_\chi>0,\qquad
 \int_X\bar P(A_{\omega_0})\,dV_\chi=0.
\]
Its diagonal restriction satisfies \eqref{strict}, and its first
derived component is $\mathcal A$. Write its monic diagonal restriction as
\[
 \frac{\bar P(s\one)}{\binom nd a_d}
 =s^d-\sum_{i=0}^{d-1}\binom di c_i s^i.
\]
For $m=d-n+q$, direct differentiation gives
\[
 \frac{n!\,\bigl(\Lambda\wedge e^\alpha\bigr)^{[q]}}{\binom nd a_d}
 =\frac{d!}{m!}[\chi]^{n-d}
 \left(\alpha^m-
 \sum_{i=0}^{m-1}\binom mi c_{n-q+i}\alpha^i[\chi]^{m-i}\right),
\]
with the sum empty when $m=0$. These are positive multiples of the
mixed classes in the numerical criterion of Fu--Li--Zhang--Zhang
\cite[Theorem~1.2]{FLZZ}. By assumption \textup{(i)}, that theorem
gives a smooth closed representative $\omega_\tau\in\alpha$ with
$A_{\omega_\tau}\in\C_{\bar P+\tau}$ for every sufficiently small
$\tau>0$. Choose $\tau<-\bar P(\rho_1\one)$. Then $\bar P+\tau$ still
satisfies \eqref{strict}, so
\[
 \C_{\bar P+\tau}
 \subset\bigcap_{i=1}^n\C_{\partial_i\bar P}
 =\mathcal A.
\]
This proves $\textup{(i)}\Longrightarrow\textup{(ii)}$.
Theorem~\ref{main} gives
$\textup{(ii)}\Longleftrightarrow\textup{(iii)}$ and uniqueness.

Finally, assume \textup{(ii)} and choose its admissible representative
$\underline\omega$. In a $\chi$-unitary coframe diagonalizing
$\underline\omega$, the coefficients of
$(\Lambda\wedge e^{\underline\omega})^{[q]}$, for $n-d\le q<n$, are
coordinate derivatives of $\widetilde P$ of order $n-q$. These are
nonzero since $a_d>0$, and positive on $\mathcal A$ by
Proposition~\ref{idealinputs}\textup{(i)}, applied to the first
derivatives of $\widetilde P$.
Compactness therefore gives a common $\delta_{\mathrm{num}}>0$ such that
\[
 (\Lambda\wedge e^{\underline\omega})^{[q]}-\delta_{\mathrm{num}}\chi^q
 \quad\text{is strongly positive}.
\]
For every K\"ahler class $\kappa$, wedging with the $(k-q)$th power of
a K\"ahler representative of $\kappa$ and integrating over $Y$ gives
\eqref{uniform-numerical}. In particular,
$\textup{(ii)}\Longrightarrow\textup{(i)}$.
\end{proof}
\section{Hodge--Riemann relations and Lorentzian polynomials}
\label{nhr:section}

We prove Theorem~\ref{mainNHR} for the form-valued equation
\eqref{general-equation}. Theorem~\ref{nhr:strict} establishes the
Hodge--Riemann pairings, and Theorem~\ref{nhr:master} constructs the
Lorentzian homogenization.
Both use a linear elliptic equation to pass from pointwise concavity
to intersection inequalities. We verify the hypotheses for the
spectral equation in Lemma~\ref{nhr:spectral-verification}.

Let $\Lambda=\sum_{k=n-d}^{n}\Lambda^{[k]}$, where $2\le d\le n$ and
each $\Lambda^{[k]}$ is a smooth closed real $(k,k)$-form on $X$. The polynomial
\begin{equation}\label{form-polynomial}
 \mathcal F(\Lambda,\eta):=\bigl(\Lambda\wedge e^{\eta}\bigr)^{[n]}
\end{equation}
is defined for real $(1,1)$-forms $\eta$ at each point. If $\Lambda$ is fixed, we write $\mathcal F(\eta):=\mathcal F(\Lambda,\eta)$ for simplicity. We make the following hypotheses.

\begin{enumerate}[label=\textup{(H\arabic*)},ref=\textup{(H\arabic*)},
 leftmargin=*]
\item\label{H1}There exists $\varepsilon_0>0$ such that,
for every $0\leq\varepsilon<\varepsilon_0$, the perturbed operator
\begin{equation}\label{variable-perturbation}
 \mathcal F_\varepsilon(\eta)
 :=\mathcal F(\eta)+\varepsilon\,\frac{\chi^n}{[\chi]^n}
\end{equation}
has the following property: at every $x\in X$, the polynomial
\[
 P_\varepsilon:A_\eta \longmapsto
 \frac{\,\mathcal F_\varepsilon(\eta)}{\chi^n/n!}
\]
is an ideal G\aa rding elliptic operator.
\item \label{H0solution} There exists a smooth closed real $(1,1)$-form $\omega$ such that $\omega\in  \C_{P_{\varepsilon}}$ for $0<\varepsilon<\varepsilon_0$ and \begin{equation}
    \mathcal F(\omega)=0.
\end{equation} We denote $\alpha=[\omega]\in H^{1,1}(X,\mathbb R)$.

\item\label{H2} At every $x$, the form $\Lambda^{[n-d]}$ is
strictly positive, that is, $\Lambda^{[n-d]}\wedge\tau>0$ for every nonzero
strongly positive $(d,d)$-form $\tau$ at $x$; and for all positive
definite $\beta_1,\ldots,\beta_{d-2}$ the quadratic form
$$(\eta,\zeta)\mapsto\eta\zeta\beta_1\cdots\beta_{d-2}\Lambda^{[n-d]}/(\chi^n/n!)$$ on
real $(1,1)$-forms at $x$ is nondegenerate.
\end{enumerate}

In particular, if $\Lambda$ gives a spectral equation as in Theorem~\ref{main} which admits a solution, then the above hypotheses hold.

\begin{lemma}\label{nhr:spectral-verification}
Let $\omega$ be an admissible solution of the spectral equation in
Theorem~\ref{main}, with $\Lambda$ given by \eqref{Lambda-def}.
Then \ref{H1}--\ref{H2} hold.
\end{lemma}

\begin{proof}
Since $\Lambda$ is given by \eqref{Lambda-def},
$\mathcal F(\eta)/(\chi^n/n!)=P_x(A_\eta)$.
Lemmas~\ref{shiftlemma} and \ref{nhr:scalarization} show that
$P_\varepsilon$ is ideal G\aa rding elliptic for all sufficiently small
$\varepsilon\ge0$. This gives \ref{H1}.
For $\varepsilon>0$, we have $A_\omega\in\mathcal A$ and
$\mathcal F_\varepsilon(\omega)>0$, so
\eqref{componentidentity} gives $A_\omega\in\C_{P_\varepsilon}$,
proving \ref{H0solution}.

Finally, $\Lambda^{[n-d]}=a_d\chi^{n-d}/(n-d)!$ is strictly positive,
and its terminal mixed pairings are nondegenerate by the mixed
Hodge--Riemann theorem \cite[Proposition~2.1]{NHRDN}. This gives
\ref{H2}.
\end{proof}

At the cohomology level, for any $[\gamma]\in H^{1,1}(X,\mathbb R)$, we define the intersection polynomial by
$$\mathcal I([\gamma]):=\int_X\mathcal F(\gamma).$$ 
Notice that $\mathcal I(\alpha)=0$ by \ref{H0solution}.

\subsection{The passage to cohomology}

We use a linear elliptic equation to choose a representative that is
pointwise primitive. This correction appears in Xiao's proof
\cite[Section~3.2]{Xiao21a}; see also Collins \cite{NHRCollins}.
Subscripts on differentials indicate evaluation at $\omega$ or $\alpha$.

For a tuple $B=(\beta_1,\ldots,\beta_j)$ of $(1,1)$-forms, write
\[
 \partial^B=\partial_{\beta_1}\cdots\partial_{\beta_j}.
\]
For the empty tuple, $\partial^B$ is the identity. The same notation
will be used for tuples of cohomology classes.

\begin{lemma}\label{nhr:pointwise}
Assume \ref{H1}--\ref{H2} and fix $0<\varepsilon<\varepsilon_0$. Let
$B=(\beta_1,\ldots,\beta_j)$ be a tuple of K\"ahler forms,
$0\le j\le d-2$, and put $m=d-j$. At each point of $X$, the polynomial
\[
 g=\frac{\partial^B\mathcal F_\varepsilon}{\chi^n/n!}
\]
is positive and strictly elliptic on a neighborhood of $\omega$ in the
space of real $(1,1)$-forms, and $g^{1/m}$ and $\log g$ are concave there.
\end{lemma}

\begin{proof}
Fix a point of $X$. By \ref{H1} and \ref{H0solution}, $P_{\varepsilon}$ is
ideal G\aa rding elliptic and $\omega\in\C_{P_{\varepsilon}}$.
The degree-$m$ homogeneous part of $g=\partial^B P_{\varepsilon}$,
evaluated at $\chi$, is
\[
 \frac{\beta_1\wedge\cdots\wedge\beta_j\wedge\chi^m
       \wedge\Lambda^{[n-d]}}{m!\,\chi^n/n!}>0
\]
by \ref{H2}. Thus $g$ has degree exactly $m$.
Lemma~\ref{ige:basic}\textup{(iii)} gives positivity of $g$ and
concavity of $g^{1/m}$ and $\log g$ on $\C_{P_{\varepsilon}}$,
which contains a neighborhood of $\omega$.

For a nonzero rank-one semipositive form $\eta$, the degree-$(m-1)$
homogeneous part of $\partial_\eta g$, evaluated at $\chi$, is
\[
 \frac{\beta_1\wedge\cdots\wedge\beta_j\wedge\eta\wedge\chi^{m-1}
       \wedge\Lambda^{[n-d]}}{(m-1)!\,\chi^n/n!}>0.
\]
Hence $\partial_\eta g$ is nonzero and positive on $\C_{P_{\varepsilon}}$
by the same lemma. This proves strict ellipticity.
\end{proof}

\begin{theorem}\label{nhr:strict}
Assume \ref{H1}--\ref{H2}. For every K\"ahler tuple $B$ of length
$0\le j\le d-2$, with $m=d-j$ and any
$0<\varepsilon<\varepsilon_0$, put
\[
 L_B=\partial(\partial^B\mathcal I)_\alpha,
 \qquad Q_B=\partial^2(\partial^B\mathcal I)_\alpha.
\]
Then $Q_B$ has signature $(1,h^{1,1}(X)-1)$ and is negative definite
  on $\ker L_B$. Moreover, for any $u\in H^{1,1}(X,\mathbb R)$, \begin{equation}\left(\int_X\partial^B\mathcal F_\varepsilon(\omega)\right)
 Q_B(u,u)\le\frac{m-1}{m}L_B(u)^2.\label{nhr:global-root}\end{equation}
\end{theorem}

\begin{proof}
Choose K\"ahler representatives of $B$, fix
$0<\varepsilon<\varepsilon_0$, and put $m=d-j$.
At each point, consider
\[
 g_k=\frac{\partial_\chi^k\partial^B\mathcal F_\varepsilon}
                {\chi^n/n!},\qquad 0\le k\le m-2.
\]
Lemma~\ref{nhr:pointwise} gives positivity and log concavity near
$\omega$ at every level. The Hessian of $g_{m-2}$ is a positive multiple of
\[
 (\eta,\zeta)\longmapsto
 \frac{\eta\zeta\beta_1\cdots\beta_j\chi^{m-2}\Lambda^{[n-d]}}
      {\chi^n/n!},
\]
which is nondegenerate by \ref{H2}.

All derivatives in the following pointwise argument are evaluated at
$\omega$. Suppose $\partial g_k[v]=\partial^2g_k[v,v]=0$.
Negative semidefiniteness of $\partial^2\log g_k$ implies
$\partial^2g_k[v,w]=0$ for every $w$, and hence
$\partial g_{k+1}[v]=0$ when $k<m-2$.
For $k<m-2$, the nonpositive function $\partial^2\log g_k[v,v]$
attains zero at $\omega$, so
\[
 0=\partial_\chi\bigl(\partial^2\log g_k[v,v]\bigr)
   =\frac{\partial^2g_{k+1}[v,v]}{g_k}.
\]
Iterating reaches the quadratic endpoint, where its nondegenerate
Hessian forces $v=0$. Thus $\partial^2g_0$ is strictly negative on
$\ker\partial g_0$.

For $u\in H^{1,1}(X,\mathbb R)$, take a closed representative $\eta$
and solve
\begin{equation}\label{nhr:corrector}
 \partial(\partial^B\mathcal F_\varepsilon)_\omega
       [\eta+\ddc\varphi]
 =\frac{L_B(u)}{\int_X\partial^B\mathcal F_\varepsilon(\omega)}
       \partial^B\mathcal F_\varepsilon(\omega).
\end{equation}
The first derivative is strictly elliptic, and its defining form is
closed. The operator in $\varphi$ is therefore self-adjoint with
kernel the constants. The equation has zero integral after subtracting
$\partial(\partial^B\mathcal F_\varepsilon)_\omega[\eta]$ from both
sides, so it has a smooth solution, unique modulo constants.

Apply the concavity of $g_0^{1/m}$ to $\eta_\varphi:=\eta+\ddc \varphi$ and use \eqref{nhr:corrector} to obtain \begin{align}   \partial^2g_0[\eta_\varphi,\eta_\varphi]&\leq\frac{m-1}{m}\frac{(\partial g_0[\eta_\varphi])^2}{g_0(\omega)}\\
    &=\frac{m-1}{m}\left(\frac{L_B(u)}{\int_X\partial^B\mathcal F_{\varepsilon}(\omega)}\right)^2\,g_0.\notag
\end{align}
The forms defining the second derivative of
$\partial^B\mathcal F_\varepsilon$ at $\omega$ are closed.  By Stokes'
theorem, integrating against $\chi^n/n!$ gives
\begin{equation}
 \left(\int_X\partial^B\mathcal F_\varepsilon(\omega)\right)
 Q_B(u,u)\le\frac{m-1}{m}L_B(u)^2.
\end{equation}

If $0\not=u\in\ker L_B$, $\eta_\varphi$ cannot vanish identically, so the strict pointwise
negativity gives $$Q_B(u,u)=\int_X \partial^2g_0(\omega)[\eta_\varphi,\eta_\varphi]\,\frac{\chi^n}{n!}<0.$$
Evaluating at $\chi$ gives $L_B([\chi])>0$ and
$Q_B([\chi],[\chi])>0$, proving the signature.


\end{proof}

\subsection{The global Lorentzian polynomial}

\begin{theorem}\label{nhr:master}
Assume \ref{H1}--\ref{H2}. Let $\beta_1,\ldots,\beta_N$ be nef
classes, $\mathbf{x}=(x_1,\ldots,x_N)$, and
$0\le\varepsilon<\varepsilon_0$. The degree-$d$ homogeneous polynomial
\begin{equation}\label{nhr:homogenization}
 \mathcal T_{\alpha,\varepsilon}(y,\mathbf{x})
 =y^d\left[\mathcal I\!\left(\alpha+
        \frac1y\sum_{i=1}^N x_i\beta_i\right)+\varepsilon\right],
\end{equation}
with its polynomial extension at $y=0$, is Lorentzian.
\end{theorem}

\begin{proof}
First suppose that $0<\varepsilon<\varepsilon_0$ and the $\beta_i$
are K\"ahler. Choose K\"ahler representatives.
Put $f(\mathbf{x})=\mathcal I(\alpha+\sum_i x_i\beta_i)+\varepsilon$.
For a multi-index $\nu$ with $|\nu|=j\le d-2$, set
$g=\partial_{\mathbf{x}}^\nu f$, $m=d-j$, and
\[
 J=g(0),\qquad b=\nabla g(0),\qquad Q_\nu=\partial^2g(0).
\]
The polynomial $g$ is the integral of the corresponding mixed
derivative of $\mathcal F_\varepsilon$.
Inequality~\eqref{nhr:global-root} gives
\[
 J>0,\qquad Q_\nu-\frac{m-1}{mJ}bb^T\le0.
\]
The Hessian of
$\partial_y^{m-2}\partial_{\mathbf{x}}^\nu\mathcal T_{\alpha,\varepsilon}$ is
\[
 (m-2)!
 \begin{pmatrix}
  m(m-1)J &(m-1)b^T\\
  (m-1)b &Q_\nu
 \end{pmatrix}.
\]
Its Schur complement is negative semidefinite, so it has exactly
one positive eigenvalue. These are all coordinate derivatives of
order $d-2$.

At each point $p\in X$, put
$h_p(\mathbf{x})=P_{\varepsilon}(A_\omega(p)+\sum_i x_iB_i)$,
where $B_i=\chi^{-1}\beta_i>0$.
By \ref{H1} and \ref{H0solution}, $h_p$ is ideal G\aa rding with the origin in its component.
Every iterated coordinate derivative of order at most $d$ is nonzero
by \ref{H2} and positive at the origin by
Proposition~\ref{idealinputs}\textup{(i)}. Thus every degree-$d$ monomial in
$y^dh_p(\mathbf{x}/y)$ has a positive coefficient. Integration over $X$
preserves positivity, including the coefficient $\varepsilon$ of $y^d$.
The support is therefore the full simplex of degree-$d$ multi-indices.
The Hessian criterion of Br\"and\'en--Huh
\cite[Theorem~2.25]{NHRBH} now shows that
$\mathcal T_{\alpha,\varepsilon}$ is Lorentzian.

For nef classes $\beta_i$ and $0\le\varepsilon<\varepsilon_0$, take
$0<\tau<\varepsilon_0-\varepsilon$ and put
$\beta_i^{(\tau)}=\beta_i+\tau[\chi]$. These classes are K\"ahler, so the polynomials
\[
 y^d\left[\mathcal I\!\left(\alpha+
       \frac1y\sum_{i=1}^N x_i\beta_i^{(\tau)}\right)+\varepsilon+\tau\right]
\]
are Lorentzian by the preceding argument.
Since the Lorentzian class is closed
\cite[Theorem~2.25]{NHRBH}, taking $\tau\to0$ proves the assertion. 
\end{proof}

This proves Theorem~\ref{mainNHR}\ref{mainNHR-b}.
The cone-Lorentzian formulation follows by choosing a finite K\"ahler
family that contains the given directions and spans $H^{1,1}(X,\mathbb R)$.

\subsection{Lefschetz isomorphisms and Hodge-index inequalities}

Set 
\begin{equation}\label{nhr:theta-class}
 \Theta=\bigl(\Lambda\wedge e^\alpha\bigr)^{[n-2]},
 \qquad Q(u,v)=\int_Xuv\Theta.
\end{equation}

\begin{corollary}\label{nhr:lefschetz}
The pairing $Q$ has signature $(1,h^{1,1}(X)-1)$, and cup product
induces an isomorphism
\begin{equation}\label{nhr:cup-isomorphism}
 \Theta\smile-:H^{1,1}(X,\mathbb R)
                   \longrightarrow H^{n-1,n-1}(X,\mathbb R).
\end{equation}
If $X$ is projective, the restriction to divisor classes has signature
$(1,\rho(X)-1)$, and the pairing induces an isomorphism
\begin{equation}\label{nhr:numerical-lefschetz}
 N^1(X)_{\mathbb R}\longrightarrow
 N^1(X)_{\mathbb R}^{*}\simeq N_1(X)_{\mathbb R},
 \qquad u\longmapsto\left(v\longmapsto\int_Xuv\Theta\right).
\end{equation}
\end{corollary}

\begin{proof}
Theorem~\ref{nhr:strict} gives the signature, and Hodge--Poincar\'e duality
gives \eqref{nhr:cup-isomorphism}. On a projective $X$, the divisor
subspace contains an ample class of positive square, so the restriction
is nondegenerate. The perfect numerical pairing identifies
$N^1(X)_{\mathbb R}^{*}$ with $N_1(X)_{\mathbb R}$, giving
\eqref{nhr:numerical-lefschetz}.
\end{proof}

\begin{corollary}\label{nhr:hodge-index}
Let $\mathcal D$ be the component of $\{\gamma\in H^{1,1}(X,\mathbb R):Q(\gamma,\gamma)>0\}$
containing $[\chi]$. For
$\gamma,\eta\in\mathcal D$,
\begin{equation}\label{nhr:AF-inequality}
 \left(\int_X\gamma\eta\,\Theta\right)^2
 \ge\left(\int_X\gamma^2\Theta\right)
     \left(\int_X\eta^2\Theta\right),
\end{equation}
with equality exactly when $\gamma$ and $\eta$ are positively
proportional. The weak inequality extends to nef classes.
\end{corollary}

\begin{proof}
Mixed positivity places the K\"ahler cone in $\mathcal D$.
For $\gamma,\eta\in \mathcal D$, write $\eta=b\gamma+\eta_0$, where
$b>0$ and $Q(\gamma,\eta_0)=0$. The orthogonal complement of $\gamma$
is negative definite, so
\[
 Q(\gamma,\eta)^2-Q(\gamma,\gamma)Q(\eta,\eta)
 =-Q(\gamma,\gamma)Q(\eta_0,\eta_0)\ge0,
\]
with equality exactly when $\eta_0=0$. Approximation by K\"ahler classes
gives the weak inequality for nef classes.
\end{proof}

\section{Examples}\label{nhr:applications}
We exhibit two examples of spectral ideal G\aa rding equations satisfying Hodge--Riemann relations.
The first example is a degree $4$ spectral equation on a fivefold
with an admissible solution in a non-nef class. Its Hessian class
satisfies the Hodge--Riemann relations but is not a product of three
nef classes. The second example is a degree $5$ spectral equation
whose Lorentzian homogenization cannot be realized as a homogeneous
volume polynomial on the same manifold.

\begin{example}[A non-nef solution class on a fivefold]
\label{nhr:fivefold}
Let
\[
 X=\mathbb P^3\times\mathbb P^1\times\mathbb P^1,
 \qquad [\chi]=A+B+C,
 \qquad \alpha=3A-\frac18B+8C,
\]
where $A,B,C$ are the hyperplane pullbacks. Then
$A^4=B^2=C^2=0$ and $\int_XA^3BC=1$.
Take
\[
 \Lambda=\frac15\chi-\frac1{10}\chi^3
         +\frac7{120}\chi^4-\frac7{320}\chi^5.
\]
The diagonal restriction
\(
 p_x(s)=s^4-6s^2+7s-\frac{21}{8}
\)
has the monotone root sequence
\[
 \rho_0>\rho_1>\rho_2=1>\rho_3=0.
\]
The intersection ring gives $\mathcal I(\alpha)=0$ and, in cohomology,
\[
 \begin{aligned}
 \Lambda\wedge e^\alpha
 &=\frac15[\chi]+\frac{24A^2+23AB+88AC+63BC}{40}\\
 &\quad+\frac{32A^3+21A^2B+216A^2C+157ABC}{40}\\
 &\quad+\frac{16A^3B+796A^3C+363A^2BC}{120}.
 \end{aligned}
\]
Each nonzero homogeneous component has a smooth strictly positive
representative, so
Corollary~\ref{cor:numerical-solvability} gives a smooth admissible
solution in the non-nef class $\alpha$.

Here
\[
 \Theta=\frac1{10}[\chi](\alpha^2-[\chi]^2).
\]
The pairing $\int_Xuv\Theta$ has signature $(1,2)$, and cup product
with $\Theta$ gives an isomorphism
$H^{1,1}(X,\mathbb R)\simeq H^{4,4}(X,\mathbb R)$.
This class is not a product of nef classes: expansion of such a
product shows that the product of its $A^2B$ and $A^2C$ coefficients
is at least the product of its $A^3$ and $ABC$ coefficients, whereas
\[
 \frac{21}{40}\frac{27}{5}-\frac45\frac{157}{40}
 =-\frac{61}{200}<0.
\]
Its signature also follows from Xiao's $4$-positive theorem
\cite[Theorem~A]{Xiao21a}.
\end{example}

\begin{example}[A non-volume homogenization]
\label{nhr:quintic}
Let
\[
 X=\mathbb P^3\times\mathbb P^2,
 \qquad [\chi]=A+B,\qquad \alpha=2A+B,
\]
where $A,B$ are the hyperplane pullbacks. Then $A^4=B^3=0$ and
$\int_XA^3B^2=1$. Take
\[
 \Lambda=1-\frac1{20}\chi^2-\frac1{60}\chi^3
          -\frac{17}{1200}\chi^5.
\]
The diagonal restriction
\[
 p_x(s)=s^5-s^3-s^2-\frac{17}{10}
\]
has the monotone root sequence
\[
 \rho_0>\rho_1=1>\rho_2>\rho_3=\frac1{\sqrt{10}}>\rho_4=0.
\]
The intersection ring gives $\mathcal I(\alpha)=0$ and, in cohomology,
\[
 \begin{aligned}
 \Lambda\wedge e^\alpha
 &=1+2A+B+\frac{39A^2+38AB+9B^2}{20}\\
 &\quad+\frac{73A^3+102A^2B+45AB^2}{60}
       +\frac{110A^3B+63A^2B^2}{120}.
 \end{aligned}
\]
Each nonzero homogeneous component has a smooth strictly positive
representative, so Corollary~\ref{cor:numerical-solvability} gives a
smooth admissible solution in $\alpha$. This equation is also covered
by  \cite{FMForms}.

Since $[\chi]$ and $\alpha$ are ample, Theorem~\ref{nhr:master} shows that
\[
 \mathcal T_{\alpha,0}(y,\mathbf{x})
 =y^5\mathcal I\!\left(\alpha+
             \frac{x_1[\chi]+x_2\alpha}{y}\right)
\]
is Lorentzian. Direct calculation gives
\[
 \det\partial^2\mathcal T_{\alpha,0}(1,0,0)
 =\frac{64387}{6000}\ne0.
\]
Every polynomial $\int_X(yu+x_1v+x_2w)^5$, with
$u,v,w\in H^{1,1}(X,\mathbb R)$, factors through the two-dimensional
space $H^{1,1}(X,\mathbb R)$ and hence has a singular Hessian.
Thus $\mathcal T_{\alpha,0}$ cannot be a homogeneous volume polynomial
on this same $X$.
\end{example}

\bibliographystyle{amsplain}
\bibliography{references}

\end{document}